\documentclass[12pt,fullpage,doublespace]{article}
\usepackage{graphicx} 
\usepackage{setspace} 
\usepackage{latexsym} 
\usepackage{amsfonts} 
\usepackage{amsmath} 
\usepackage{amssymb}
\usepackage{textcomp}
\usepackage{undertilde}
\usepackage{bm}
\usepackage{stmaryrd}
\usepackage[margin=1in]{geometry}
\usepackage{amsthm}
\usepackage{ifsym}
\usepackage{amssymb,latexsym,amsmath}
\usepackage{graphics}

\newtheorem{theorem}{Theorem}[section]

\newtheorem{definition}[theorem]{Definition}

\newtheorem{lemma}[theorem]{Lemma}
\newtheorem{corollary}[theorem]{Corollary}

\numberwithin{equation}{section}

\newcommand{\rest}{\restriction}

\def\a{\alpha}

\renewcommand{\models}{\vDash}
\newcommand{\powerset}{{\cal P}}

\def\P{{\mathcal{P} }}

\def\Q{{\mathcal{ Q}}}

\def\K{{\mathcal{ K}}}
\def\L{{\rm{L}}}

\def\R{{\mathcal R}}

\def\M{{\mathcal{M}}}
\def\N{{\mathcal{N}}}
\def\F{{\mathcal{F}}}

\def\S{{\mathcal{S}}}
\def\F{{\mathcal{F}}}

\newcommand{\rthm}[1]{Theorem~\ref{#1}}
\newcommand{\rlem}[1]{Lemma~\ref{#1}}

\newcommand{\insegeq}{\trianglelefteq}
\newcommand{\inseg}{\triangleleft}

 \input xy
 \xyoption{all}
\begin{document}
\title{Determinacy in $L(\mathbb{R},\mu)$}
\date{\today}
\author{Nam Trang\\
Department of Mathematical Sciences\\
Carnegie Mellon University\\
namtrang@andrew.cmu.edu}
\maketitle
\begin{abstract}
Assume $V=L(\mathbb{R},\mu) \vDash \textsf{ZF} + \textsf{DC} + \Theta > \omega_2 + \mu $ is a normal fine measure on $\powerset_{\omega_1}(\mathbb{R})$. We analyze what sets of reals are determined and in fact show that $L(\mathbb{R},\mu) \vDash \textsf{AD}$. This arguably gives the most optimal characterization of \textsf{AD} in $L(\mathbb{R},\mu)$. As a consequence of this analysis, we obtain the equiconsistency of the theories: ``\textsf{ZFC} + There are $\omega^2$ Woodin cardinals" and ``$\textsf{ZF} + \textsf{DC} + \Theta>\omega_2 \ + $ There is a normal fine measure on $\powerset_{\omega_1}(\mathbb{R})$".
\end{abstract}
\thispagestyle{empty}
\thispagestyle{empty}

\section{Introduction}
A measure $\mu$ on $\powerset_{\omega_1}(\mathbb{R})$ is \textbf{fine} if for all $x\in\mathbb{R}$, $\mu(\{ \sigma \in \powerset_{\omega_1}(\mathbb{R}) \ | \ x\in \sigma\}) = 1$. $\mu$ is \textbf{normal} if for all functions $F: \powerset_{\omega_1}(\mathbb{R})\rightarrow \powerset_{\omega_1}(\mathbb{R})$ such that $\mu(\{ \sigma \ | \ F(\sigma)\subseteq \sigma\}) = 1$, there is an $x\in \mathbb{R}$ such that $\mu(\{\sigma \ | \ x \in F(\sigma)\}) = 1$. A normal fine measure $\mu$ on $\powerset_{\omega_1}(\mathbb{R})$ is often called the Solovay measure. Solovay (in \cite{solovay1978independence}) has shown the existence (and uniqueness) of a normal fine measure on $\powerset_{\omega_1}(\mathbb{R})$ under $\textsf{AD}_\mathbb{R}$. It is natural to ask whether the existence of a normal fine measure on $\powerset_{\omega_1}(\mathbb{R})$ has consistency strength that of $\textsf{AD}_\mathbb{R}$. It is well-known that the existence of an $L(\mathbb{R},\mu)\footnote{By $L(\mathbb{R},\mu)$ we mean the model constructed from the reals and using $\mu$ as a predicate. We will also use the notation $L(\mathbb{R})[\mu]$ and $L_{\alpha}(\mathbb{R})[\mu]$ in various places in the paper.}$ that satisfies ``$\rm{\textsf{ZF} + \textsf{DC} \ + }\ \mu$ is a normal fine measure on $\powerset_{\omega_1}(\mathbb{R})$" is equiconsistent with that of a measurable cardinal; this is much weaker than the consistency strength of $\textsf{AD}_\mathbb{R}$. The model $L(\mathbb{R},\mu)$ obtained from standard proofs of the equiconsistency satisfies $\Theta\footnote{$\Theta$ is the sup of all $\alpha$ such that there is a surjection from $\mathbb{R}$ onto $\alpha$} = \omega_2$ and hence fails to satisfy \textsf{AD}. So it is natural to consider the situations where $L(\mathbb{R},\mu) \vDash \Theta > \omega_2$ and try to understand how much determinacy holds in this model; furthemore, one can try to ask what the exact consistency strength of the theory ``$\Theta > \omega_2$ + there is a normal fine measure on $\powerset_{\omega_1}(\mathbb{R})$" is. What about the seemingly stronger theory ``$\textsf{AD}$ + there is a normal fine measure on $\powerset_{\omega_1}(\mathbb{R})$"?
\\
\indent We attempt to answer some of the above questions in this paper. First, to analyze the sets of reals that are determined in such a model, which we will call $V$, we run the core model induction in a certain submodel of $V$ that agrees with $V$ on all bounded subsets of $\Theta$. This model will be defined in the next section. What we'll show is that $K(\mathbb{R}) \vDash \textsf{\textsf{AD}}^+$ where 
\begin{equation*}
K(\mathbb{R}) = L(\cup\{ \M \ | \ \M \textrm{ is an } \mathbb{R}\textrm{-premouse, } \rho(\M) = \mathbb{R}, \textrm{ and } \M \textrm{ is countably iterable}\footnote{An $\mathbb{R}$-premouse $\M$ is countably iterable if any countable hull of $\M$ is $\omega_1+1$ iterable.}\}).
\end{equation*}
We will then show $\Theta^{K(\mathbb{R})} = \Theta$ by an argument like that in Chapter 7 of \cite{CMI}. Finally, we prove that 
\begin{equation*}
\powerset(\mathbb{R}) \cap K(\mathbb{R}) = \powerset(\mathbb{R}),
\end{equation*}
which implies $L(\mathbb{R},\mu) \vDash \textsf{\textsf{AD}}^+$. We state the main result of this paper.
\begin{theorem}
\label{characterization of AD}
Suppose $V = L(\mathbb{R},\mu) \vDash \rm{\textsf{ZF} + \textsf{DC} \ + }\ \Theta > \omega_2 + \mu$ is a normal fine measure on $\powerset_{\omega_1}(\mathbb{R})$. Then $L(\mathbb{R},\mu) \vDash \textsf{\textsf{AD}}^+$.
\end{theorem}
Woodin has shown the following.
\begin{theorem}
\label{mu unique}
Suppose $L(\mathbb{R},\mu) \vDash \textsf{\textsf{AD}}\ + \mu$ is a normal fine measure on $\powerset_{\omega_1}(\mathbb{R})$. Then $L(\mathbb{R},\mu) \vDash \textsf{\textsf{AD}}^+ + \mu$ is unique.
\end{theorem}
Combining the results in Theorem \ref{mu unique} and Theorem \ref{characterization of AD}, we get the following.
\begin{corollary}
\label{ADmuunique}
Suppose $V = L(\mathbb{R},\mu) \vDash \rm{\textsf{ZF} + \textsf{DC} \ + }\ \Theta > \omega_2 + \mu$ is a normal fine measure on $\powerset_{\omega_1}(\mathbb{R})$. Then $L(\mathbb{R},\mu) \vDash \textsf{\textsf{AD}}^+ + \mu$ is unique.
\end{corollary}
The equiconsistency result we get from this analysis is stated in Theorem \ref{Equiconsistency}. We should mention that these theories are still much weaker than $\textsf{AD}_\mathbb{R}$ in consistency strength.
\begin{theorem}
\label{Equiconsistency}
The following theories are equiconsistent.
\begin{enumerate}
\item $\textsf{\textsf{\textsf{ZF}C}} \ + $ There are $\omega^2$ Woodin cardinals.
\item $\rm{\textsf{ZF} + \textsf{DC} + \textsf{AD}^+} \ + $ There is a normal fine measure on $\powerset_{\omega_1}(\mathbb{R})$.
\item $\rm{\textsf{ZF} + \textsf{DC} \ + }\ \Theta > \omega_2 \ + $ There is a normal fine measure on $\powerset_{\omega_1}(\mathbb{R})$.  
\end{enumerate}
\end{theorem}
\begin{proof}
The equiconsistency of (1) and (2) is a theorem of Woodin (see \cite{Woodin} for more information). Theorem \ref{characterization of AD} immediately implies the equiconsistency of (2) and (3).
\end{proof}
We would like to thank Hugh Woodin for suggesting this problem to us, his encouragement and insightful discussions on the subject matter. We would also like to thank John Steel and Martin Zeman for their helpful comments at various stages of the project. Part of this result was proved during the author's stay in Singapore in Summer 2011; we would like to thank IMS of NUS for their hospitality.

\section{Basic setup}
\label{basicsetup}
In this section we prove some basic facts about $V$ assuming $V = L(\mathbb{R},\mu) \models \rm{\textsf{ZF} + \textsf{DC} \ +\ }\mu$ is a normal fine measure on $\powerset_{\omega_1}(\mathbb{R})$. First note that we cannot well-order the reals hence full \textsf{AC} fails in this model. Secondly, $\omega_1$ is regular; this follows from \textsf{DC}. Now $\mu$ induces a countably complete nonprincipal ultrafilter on $\omega_1$; hence, $\omega_1$ is a measurable cardinal. \textsf{DC} also implies that cof$(\omega_2) > \omega$. We collect these facts into the following lemma, whose easy proof is left to the reader.
\begin{lemma}
\label{facts}
Suppose $V = L(\mathbb{R},\mu) \models \rm{\textsf{ZF} + \textsf{DC} \ +\ } \mu$ is a normal fine measure on $\powerset_{\omega_1}(\mathbb{R})$. Then
\begin{enumerate}
\item $\omega_1$ is regular and in fact measurable; 
\item cof$(\omega_2) > \omega$;
\item $\rm{\textsf{AC}}$ fails and in fact, there cannot be an $\omega_1$-sequence of distinct reals.
\end{enumerate}
\end{lemma}

\begin{lemma}
\label{regularity of theta}
$\Theta$ is a regular cardinal.
\end{lemma}
\begin{proof}
Suppose not. Let $f: \mathbb{R} \rightarrow \Theta$ be a cofinal map. Then there is an $x \in \mathbb{R}$ such that $f$ is $\mathrm{OD}(\mu,x)$. For each $\alpha < \Theta$, there is a surjection $g_\alpha: \mathbb{R} \rightarrow \alpha$ such that $g_\alpha$ is $\mathrm{OD}(\mu)$ (we may take $g_\alpha$ to be the least such). We can get such a $g_\alpha$ because we can ``average over the reals." Now define a surjection $g: \mathbb{R} \rightarrow \Theta$ as follows
\begin{equation*}
g(y) = g_{f(y_0)}(y_1) \textrm{ where \ } y = \langle y_0, y_1 \rangle.
\end{equation*}
It's easy to see that $g$ is a surjection. But this is a contradiction.
\end{proof}
\begin{lemma}
\label{omega 1 inaccessible}
$\omega_1$ is inaccessible in any (transitive) inner model of choice containing $\omega_1$.
\end{lemma}
\begin{proof}
This is easy. Let $N$ be such a model. Since $P = L(N,\mu)$ is also a choice model and $\omega_1$ is measurable in $P$, hence $\omega_1$ is inaccessible in $P$. This gives $\omega_1$ is inaccessible in $N$.
\end{proof}
\indent Next, we define two key models that we'll use for our core model induction. Let 
\begin{equation*}
M = \Pi_{\sigma \in \powerset_{\omega_1}(\mathbb{R})} M_\sigma\slash \mu \textrm{ where \ } M_\sigma = \rm{\rm{H\mathrm{OD}}}_{\sigma \cup \{\sigma\}} 
\end{equation*}
and, for a transitive self-wellordered\footnote{This means there is a well-ordering of $a$ in $L_1[a]$.} $a\in H_{\omega_1}$,
\begin{equation*}
H_a = \Pi_{\sigma \in \powerset_{\omega_1}(\mathbb{R})} H_{a,\sigma}\slash \mu \textrm{ where \ } H_{a,\sigma} = \rm{\rm{H\mathrm{OD}}}^{M_\sigma}_{\{a,\sigma\}}. 
\end{equation*}
We note that in the definition of $M_\sigma$ and $H_\sigma$ above, ordinal definability is with respect to the structure $(L(\mathbb{R},\mu),\mu)$.
\begin{lemma}
\label{Los holds}
$\L\acute{o}s$ theorem holds for both of the ultraproducts defined above.
\end{lemma}
\begin{proof}
We do this for the first ultraproduct. The proof is by induction on the complexity of formulas. It's enough to show the following. Suppose $\phi(x,y)$ is a formula and $f$ is a function such that $\forall^*_\mu \sigma M_\sigma \vDash \exists x \phi[x, f(\sigma)]$. We show that $M \vDash \exists x \phi[x,[f]_\mu)$. 
\\
\indent Let $g(\sigma) = \{ x \in \sigma \ | \ (\exists y \in \mathrm{OD}(\mu,x)) (M_\sigma \vDash \phi[y,f(\sigma]) \}$. Then $\forall^*_\mu \sigma g(\sigma)$ is a non-empty subset of $\sigma$. By normality of $\mu$, there is a fixed real $x$ such that $\forall^*_\mu \sigma x \in g(\sigma)$. Hence we can define $h(\sigma)$ to be the least $y$ in $\mathrm{OD}(\mu,x)$ such that $M_\sigma \vDash \phi[y,f(\sigma)]$. It's easy to see then that $M \vDash \phi[[h]_\mu, [f]_\mu]$.  
\end{proof}
By Lemma \ref{Los holds}, $M$ and $H_a$ are well-founded so we identify them with their transitive collapse. First note that $M \vDash$ \textsf{ZF} + \textsf{DC} and $H_a \vDash$ \textsf{\textsf{ZFC}}. We then observe that $\Omega = [\lambda \sigma. \omega_1]_\mu$ is measurable in $M$ and in $H_a$. This is because $\omega_1$ is measurable in $M_\sigma$ and $H_{a,\sigma}$ for all $\sigma$. Note also that $\Omega > \Theta$ as $\forall^*_\mu \sigma$, $\Theta^{M_\sigma}$ is countable and $\powerset(\omega_1)^{M_\sigma}$ is countable. The key for this is just an easy fact stated in Lemma \ref{facts}: There are no sequences of $\omega_1$ distinct reals. Hence, by a standard Vopenka argument, for any set of ordinals $A \in M$ of size less than $\Omega$, there is an $H_a$-generic $G_A$ (for a forcing of size smaller than $\Omega$) such that $A \in H_a[G_A] \subseteq M$ and $\Omega$ is also measurable in $H_a[G_A]$. 
\begin{lemma}
\label{M contains all sets of reals}
$\powerset(\mathbb{R}) \subseteq M$.
\end{lemma}
\begin{proof}
Let $A \subseteq \mathbb{R}$. Then there is an $x\in \mathbb{R}$ such that $A \in \mathrm{OD}(x,\mu)$. By fineness of $\mu$, $(\forall^*_\mu \sigma)(x\in \sigma)$ and hence $(\forall^*_\mu \sigma)(A \cap \sigma \in \mathrm{OD}(x,\mu,\sigma))$. So we have $(\forall^*_\mu \sigma)(A\cap\sigma \in M_\sigma)$. This gives us that $A = [\lambda \sigma. A\cap \sigma]_\mu \in M$.
\end{proof}
Lemma \ref{M contains all sets of reals} implies that $M$ contains all bounded subsets of $\Theta$.
\section{Framework for the core model induction}
This section is an adaptation of the framework for the core model induction developed in \cite{sargsyan2012non}, which in turns builds on earlier formulations of the core model induction in \cite{CMI}. For a detailed discussion on basic notions such as model operators, mouse operators, $F$-mice, $Lp^F$, $Lp^\Gamma$, condenses well, relativizes well, the envelope of an inductive-like pointclass $\Gamma$ (denoted Env$(\Gamma)$), iterability, quasi-iterability, see \cite{wilson2012contributions}. We briefly recall some of these notions here. 
\begin{definition}
\label{model}
Let $\mathcal{L}_0$ be the language of set theory expanded by unary predicate
symbols $\dot{E}, \dot{B}, \dot{S}$, and constant symbols $\dot{l}$ and
$\dot{a}$. Let $a$ be a given transitive set. A \textbf{model with paramemter a}
is an $\mathcal{L}_0$-structure of the form
\begin{center}
$\mathcal{M} = (M; \in, E, B, \mathcal{S}, l, a)$
\end{center}
such that $M$ is a transtive rud-closed set containing $a$, the structure
$\mathcal{M}$ is amenable, $\dot{a}^\M = a$, $\mathcal{S}$ is a sequence of
models $\S_\xi$'s with paramemter $a$ such that letting $M_\xi$ be the universe of
$\mathcal{S}_\xi$
\begin{itemize}
\item $\dot{S}^{\mathcal{S}_\xi} = \S\rest \xi$ for all $\xi\in
\textrm{dom}(\S)$ and $\dot{S}^{\S_\xi}\in M_\xi$ if $\xi$ is a successor
ordinal;
\item $M_\xi = \cup_{\alpha<\xi}M_\alpha$ for all limit $\xi \in
\textrm{dom}(\S)$;
\item if $\textrm{dom}(\S)$ is a limit ordinal then $M = \cup_{\alpha\in
\textrm{dom}(\S)} M_\alpha$ and $l=0$, and
\item if $\textrm{dom}(\S)$ is a successor ordinal, then $\textrm{dom}(\S) =
l$.
\end{itemize}
\end{definition}
Typically, the predicate $\dot{E}$ codes the top extender of the model;
$\dot{S}$ records the sequence of models being built so far; $\dot{B}$ codes the ``lower part extenders". Next, we write down
some notations regarding the above definition.
\begin{definition}
\label{someNotations}
Let $\M$ be the model with parameter $a$. Then $|\M|$ denotes the universe of
$\M$. We let $l(\M) = dom(\dot{S}^\M)$ denote the \textbf{length of} $\M$ and
set $\M|\xi = 	\dot{S}^{\M_\xi}$ for all $\xi <l(\M)$. We set $\M|l(\M) = \M$. If $l(\M) = \xi+1$, then we let $\M^- = \dot{\S}^\M_\xi$. We
also let $\rho(\M)\leq l(\M)$ be the least such that there is some $A\subseteq
M$ definable (from parameters in $M$) over $\M$ such that $A\cap
|\M|\rho(\M)|\notin M$.
\end{definition}
\begin{definition}
\label{ModelOperator}
Let $\nu$ be an uncountable cardinal and $a\in H_\nu$ be transitive. A \textbf{model operator over $a$} is a partial function $F: H_\nu \rightarrow H_\nu$ such that to each model $\M$ over $a$, $F(\M)$ is a model over $a$ such that 
\begin{itemize}
\item $\dot{E}^{F(\M)} = \emptyset$;
\item $\dot{S}^{F(\M)} = (\dot{S}^\M)^\smallfrown \langle\M\rangle$; 
\item $F(\M) = Hull^{F(\M)}_{\Sigma_1}(|\M|)$ (here the hull is transitively collapsed);
\item if $x\in |F(\M)|$ and $y\in |M|\rho(\M)|$ then $x\cap y \in |\M|$.
\end{itemize} 
\end{definition}
For a transitive set $a$, we let $Lp(a)$ be the union of $\N$ such that $\N$ is a sound premouse over $a$, $\rho_{\omega} = a$, and for all $\pi: \bar{\N}\rightarrow \N$ such that $\bar{\N}$ is countable transitive, $\pi$ is elementary, then $\bar{\N}$ is $(\omega,\omega_1+1)$-iterable, that is $\bar{\N}$ has an iteration strategy (in fact a unique one) that acts on $\omega$-maximal, normal iteration trees of length at most $\omega_1$ on $\bar{\N}$. Let $\nu$ be an uncountable cardinal, $a\in H_\nu$ be transtive. We say that $J$ is a \textbf{mouse operator on $H_\nu$ over $a$} if there is an $rQ$-formula $\varphi(v_0,v_1)$ in the language of (Mitchell-Steel) premice such that for all transitive $b\in H_\nu$ such that $a\in b$, $J(b)$ is the least $\N\lhd Lp(b)$ such that $\N\vDash \varphi[b,a]$. We say that $J$ is defined on $H_\nu$ over $a$ or on a cone on $H_\nu$ above $a$.
\begin{definition}
\label{ModelOpInducedByMouseOp}
Let $J$ be a mouse operator on $H_\nu$ over some transitive $a\in H_\nu$. The model operator $F_J$ induced by $J$ is defined as follows:
\begin{enumerate}
\item If $J(\M)$ is amenable to $\M|\rho(\M)$ then 
\begin{center}
$F_J(\M) = (|J(\M)|,\in, \emptyset,B,(\dot{\S}^\M)^\smallfrown \langle \M \rangle, dom(\dot{\S}^\M)+1, a)$,
\end{center}
where $B$ is the extender sequence for $J(\M)$.
\item Otherwise, let $\xi$ be the least ordinal such that $J(\M)|(\xi+1)$ is not amenable to $\M|\rho(\M)$ and $n$ be the smallest such that $\rho_{n+1}(J(\M)|\xi)= \M$. Then letting $(N,\empty, B)$ be the $n^{th}$-reduct of $J(\M)|\xi$, we set
\begin{center}
$F_J(\M) = (N, \in, \emptyset, B, (\dot{\S}^\M)^\smallfrown \langle \M \rangle, dom(\dot{\S}^\M)+1, a)$.
\end{center} 
\end{enumerate}
\end{definition}

We note in the above that $\rho_1(F_J(\M)) = \M$ and $F_J(\M) = Hull_{\Sigma_1}^{F_J(\M)}(|\M|)$, and hence $F_J$ is indeed a model operator over $a$ on $H_\nu$. Sometimes when the domain of $F_J$ (or any model operator $F$) is clear or is not important, we just say that $F_J$ (or $F$) is a model operator over $a$.

\begin{definition}
\label{CondensesWell}
Suppose $F$ is a model operator over some transitive set $a$. $F$ \textbf{condenses well} if the following hold:
\begin{enumerate}
\item If $\M,\M', \N$ are models over $a$ such that $\M'=\M^-$, where $dom(\dot{S}^{\M'})$ is a successor ordinal, and $\pi: \M\rightarrow \F(\N)$ is a $0$-embedding or $\Sigma_2$-embedding then $\M = F(\M')$.
\item If $\P, \M, \M', \N$ are models with $\M' = \M^-$, $\sigma: F(\P)\rightarrow \M$ being a $0$-embedding, and $\pi: \M\rightarrow F(\N)$ being a weak $0$-embedding, then $\M = F(\M')$.
\end{enumerate}
\end{definition}
Our definition is weaker than Definition 2.1.10 \cite{wilson2012contributions} in that we don't require (1) and (2) above to hold in $V^{Col(\omega,\M)}$; our core model induction will not occur in a generic extension of $V$ hence this is all we need out of the notion of ``condenses well". Nevertheless, we can still define the notions of $F$-premice, projecta, standard parameters, solidity and universality of stadard parameters, iteration trees and stategies for $F$-premice, and the $K^{c,F}$-construction the same way as in \cite{wilson2012contributions}.  
\begin{definition}
\label{RelativizeWell}
Suppose $F$ is a model operator on $H_\nu$ over some transitive $a\in H_\nu$. We say that $F$ \textbf{relativizes well} if there is a formula $\varphi$ in the language of set theory such that for every pair $\P, \Q$ of models over $a$ such that $\P\in \Q$ and if $M$ is a transitive model of $\textsf{ZF}^-$ such that $F(\Q)\in M$, then $F(\P)$ is the unique $x\in M$ such that $M \vDash \varphi[x,\P,\Q,F(\Q)]$.  
\end{definition}
\begin{definition}
\label{DetSelfGenExt}
Suppose $F$ is a model operator on $H_\nu$ over some transitive $a\in H_\nu$. We say that $F$ \textbf{determines itself on generic extensions} if there is a formula $\varphi$ (in the language of set theory) such that whenever $M\vDash \textsf{ZF}^-$ is transitive and is closed under $F$ and $g\in V$ is generic over $M$, then $M[g]$ is closed under $F$ and $\varphi$ defines $F\rest M[g]$ over $M[g]$ from $F\rest M$.
\end{definition}

Definitions \ref{CondensesWell}, \ref{RelativizeWell}, and \ref{DetSelfGenExt} have obvious analogues for mouse operators. The model operators that we encounter during the core model induction in this paper come from mouse operators that condense well, relativize well, and determine themselves on generic extensions. We list examples of such operators. These operators, for the purpose of this paper, are defined on $H_{\omega_1}$ above some transitive $a\in H_{\omega_1}$.
\begin{enumerate}
\item $F = F_J$ for some mouse operator $J$ defined on $H_{\omega_1}$ over some $a\in H_{\omega_1}$. Some examples of $J$ are the $\M_n^\sharp$ operators, the ``diagonal operator" defined in 4.2 of \cite{CMI}, and the $\mathcal{A}$-mouse operator $J = J_{\mathcal{A}}$ defined in Definition 4.3.9 of \cite{wilson2012contributions}, where $\mathcal{A}=(A_i: i<\omega)$ is a self-justifying-system such that $\mathcal{A}\in \mathrm{OD}_{b, \Sigma, x}^{K(\mathbb{R})|\alpha}$ for some $x\in b$ and $\a$ ends either a weak gap or a strong gap in the sense of \cite{K(R)} and $\mathcal{A}$ seals the gap\footnote{This means that $\mathcal{A}$ is cofinal in $\utilde{Env}(\Gamma)$, where $\Gamma = \Sigma_1^{K(\mathbb{R})|\alpha}$.}.   
\item  For some $H$, $H$ satisfies 1) above, for some $n<\omega$, $F=F_J$, where $J$ is the $x\rightarrow \M^{\#, H}_n(x)$ operator.
\item $H=F_J$, where for some $a\in HC$ and $\M\lhd Lp(a)$, letting $\Lambda$ is $\M$'s unique $(\omega,\omega_1)$-iteration strategy (in fact $(\omega,\omega_1+1)$-strategy by using the measure $\mu$) or $\M$ is a hod premouse (in the sense of \cite{ATHM}) and $\Lambda$ is $\M$'s $(\omega,\omega_1,\omega_1)$-iteration strategy with branch condensation, for some $rQ$-formula $\varphi$, for some $b\in H_{\omega_1}$ coding $a$, for all $x\in H_{\omega_1}$ coding $b$, $J(b)\lhd Lp^\Lambda(b)$ is the least that satisfies $\varphi[x,a]$. Typically, the $H$ we need are those that are definable over $L^{\Lambda}_\alpha(\mathbb{R})$\footnote{$L^\Lambda(\mathbb{R})$ is the smallest $\textsf{ZF}$ model containing the reals, the ordinals, and is closed under $\Lambda$. See \cite{trang2012scales} for a constructive definition of $L^\Lambda(\mathbb{R})$ and also for a proof that the operators $x\mapsto \M^{\sharp,\Lambda}_n(x)$ condense well, relativize well, and determine themselves on generic extensions.} for some $\alpha$. 
\end{enumerate}
We remark that the model operators listed above are what we need to satisfy the hypothesis of Theorem \ref{the cmi theorem} to prove $\textsf{AD}$ holds in $K(\mathbb{R})$. These are what's called cmi-operators in \cite{sargsyan2012non}.

The following lemmata are what we need to propagate ``nice" properties of a model operator $F$ to those of $\M_1^{\sharp,F}$ and to $F_{\M_1^{\sharp,F}}$. These propagations are needed in the core model induction. The proofs of these lemmata are easy and hence will be ommited.
\begin{lemma}
\label{Propagation1}
Suppose $F$ is a model operator defined on $H_\nu$ over $a$. Suppose $F$ condenses well, relativizes well, and determines itself on generic extensions, then so does the mouse operator $\M_1^{\sharp,F}$ and $dom(\M_1^{\sharp,F})=dom(F)$.
\end{lemma}
\begin{lemma}
\label{Propagation2}
Suppose $J$ is a mouse operator defined on $H_\nu$ over $a$. Suppose $J$ condenses well, relativizes well, and determines itself on generic extensions, then so does the model operator $F_J$ and $dom(J)= dom(F_J)$.
\end{lemma}
\section{\textsf{AD} in $K(\mathbb{R})$}
\label{OneMoreWoodin}
By the discussion of the last section, to show \textsf{AD} holds in $K(\mathbb{R})$, it is enough to show that if $F$ is a model operator defined on a cone on $H_{\omega_1}$ above some $a\in H_{\omega_1}$ that condenses well, relativizes well, and determines itself on generic extensions, then $\M_1^{\sharp,F}(x)$ exists (and is $(\omega_1,\omega_1)$-iterable) for all $x\in H_{\omega_1}$ such that $a\in x$. By Lemmata \ref{Propagation1} and \ref{Propagation2}, then the mouse operator $\M_1^{\sharp,F}$ and the model operator $F_{\M_1^{\sharp,F}}$ condense well, relativize well, and determine themselves on generic extensions. Hence the same conclusion holds for the corresponding model operator $F_{\M_1^{\sharp,F}}$.
\begin{theorem}
\label{PD}
Suppose $F$ is a nice model operator, where $F$ is nice if it relativizes well, condenses well, and determines itself on generic extensions. Suppose $F$ is defined on the cone on $H_{\omega_1}$ above some $a\in H_{\omega_1}$. Then $\M_1^{\sharp,F}(x)$ exists for all $x\in H_{\omega_1}$ coding $a$. Furthermore, $\M_1^{\sharp,F}(x)$ is $(\omega,\omega_1,\omega_1)$-iterable, hence $(\omega,\Omega,\Omega)$-iterable in $M$.
\end{theorem}
\begin{proof}
To start off, we may assume there is some real $x\in a$ such that $F$ is OD$_{x}$, hence \begin{center}
$\forall^*_\mu \sigma \ F\cap M_\sigma \in M_\sigma$.
\end{center}
By $\L\acute{o}s$ theorem, $[\sigma\mapsto F\cap M_\sigma]$ defines a unique model operator on $H^M_\Omega$ over $a$ extending $F$ that condenses well, relativizes well, and determines itself on generic extensions. We also call this extension $F$. Now $F=F_J$ for some mouse operator $J$, by the same argument, the $F^\sharp$-operator, where $F^\sharp(x)$ is the model operator corresponds to the first active level of $Lp^{F}(x)$, is nice and is defined on the cone of $H^M_\Omega$ above $a$.
\begin{lemma}
\label{M1 sharp exists}
For each $x \in \mathbb{R}$ coding $a$, $\M_1^{\sharp,F}(x)$ exists.
\end{lemma}
\begin{proof}
This is the key lemma. Suppose not, there is some $x$ such that $\M_1^{\sharp,F}(x)$ doesn't exist. Then in $H_x$\footnote{Technically, we should take a self-wellordered transitive $b$ coding $x$ and consider $H_b$.}, which is closed under $F^\sharp$, the core model $K =_{\textrm{def}} K^F(x)$\footnote{$K^F(x)$ is the core model that comes from the relative-to-$F$ $K^{c,F}$-construction over $x$ as defined in \cite{CMI}.} (built up to $\Omega$) exists and is $(\omega,\Omega+1)$ iterable by the unique $F^\sharp$-guided strategy\footnote{This means that for any normal, $\omega$-maximal tree $\mathcal{T}\in H_x$ of length at most $\Omega$ on $K$, there is a unique cofinal wellfounded branch $b$ such that $\Q(b,\mathcal{T})\trianglelefteq F^\sharp(\mathcal{T})$.}. Let $\kappa = \omega^V_1$. By Lemma \ref{omega 1 inaccessible}, $\kappa$ is inaccessible in $H_x$ and in any $<\Omega$-set generic extension $J$ of $H_x$ such that $J \subseteq M$. By \cite{CMIP}, $K^{H_x} = K^{H_x[G]}$ for any $H_x$-generic $G$ for a poset of size smaller than $\Omega$.
\\
\\
\textbf{Claim.} $(\kappa^+)^K = (\kappa^+)^{H_x}$.
\begin{proof}
The proof follows that of Theorem 3.1 in \cite{MaxCM}. Suppose not. Let $\lambda = (\kappa^+)^K$. Hence $\lambda < (\kappa^+)^{H_x}$. Working in $H_x$, let $N$ be a transitive, $F^\sharp$-closed, power admissible set containing $x$ such that $^\omega N \subseteq N,$ $V_\kappa \cup \mathcal{J}^K_{\lambda+1} \subseteq N$, and card$(N) = \kappa$. We then choose $A\subseteq \kappa$ such that $N \in L^{F^\sharp}[A]$\footnote{Technically, we should write $L^{F^\sharp}_{\Omega}[A]$ since $F^\sharp$ is only defined on $H_\Omega^M$. All the computations in this claim will be below $\Omega$.} and $K^{L^{F^\sharp}[A]}|\lambda = K|\lambda$, $\lambda = (\kappa^+)^{K^{L^{F^\sharp}[A]}}$, and card$(N)^{L^{F^\sharp}[A]} = \kappa$. Such an $A$ exists by Lemma 3.1.1 in \cite{MaxCM} and the fact that $\lambda < (\kappa^+)^{H_x}$.
\\
\indent Now, it's easy to see that the sharp of $L^{F^\sharp}[A]$ exists in $H_x$, and hence $(\kappa^+)^{L^{F^\sharp}[A]} < (\kappa^+)^{H_x}$. By $\textsf{GCH}$ in $L^{F^\sharp}[A]$, card$^{H_x}(\powerset(\kappa)\cap L^{F^\sharp}[A]) = \kappa$. So in $M$, there is an $L^{F^\sharp}[A]$-ultrafilter $U$ over $\kappa$ that is nonprincipal and countably complete (in $M$ and in $V$). This is because such a $U$ exists in $V$ as being induced from $\mu$ and since $U$ can be coded as a subset of $\omega^V_1 = \kappa$, $U \in M$. Let $J$ be a generic extension of $H_x$ (of size smaller than $\Omega$) such that $U \in J$. From now on, we work in $J$. Let 
\begin{equation*}
j: L^{F^\sharp}[A] \rightarrow Ult(L^{F^\sharp}[A],U) \cong L^{F^\sharp}[j(A)]
\end{equation*}
be the ultrapower map. We note that since $U$ comes from $\mu$, $U$ moves $F^\sharp$ to itself; that's why it's justified to write $Ult(L^{F^\sharp}[A],U) \cong L^{F^\sharp}[j(A)]$ in the above. Then crit$(j) = \kappa$, $A = j(A) \cap \kappa \in L^{F^\sharp}[j(A)]$. So $L^{F^\sharp}[A] \subseteq L^{F^\sharp}[j(A)]$. The key point here is that $\powerset(\kappa)\cap K^{L^{F^\sharp}[A]} = \powerset(\kappa)\cap K^{L^{F^\sharp}[j(A)]}$. To see this, first note that the $\subseteq$ direction holds because any $\kappa$-strong mouse (in the sense of \cite{CMIP}) in $L^{F^\sharp}[A]$ is a $\kappa$-strong mouse in $L^{F^\sharp}[j(A)]$ as $\mathbb{R}\cap L^{F^\sharp}[A] = \mathbb{R}\cap L^{F^\sharp}[j(A)]$ and $L^{F^\sharp}[A]$ and $L^{F^\sharp}[j(A)]$ have the same $<\kappa$-strong mice. To see the converse, suppose not. Then there is a sound mouse $\M \triangleleft K^{L^{F^\sharp}[j(A)]}$ such that $\M$ extends $K^{L^{F^\sharp}[A]}|\lambda$ and $\M$ projects to $\kappa$. The iterability of $\M$ is absolute between $J$ and $L^{F^\sharp}[j(A)]$, by the following folklore result
\begin{lemma}
\label{0-absoluteness iterability}
Let $F$ be a nice model operator defined on all of $V$. Assume $\textsf{\textsf{\textsf{ZF}C}} \ + $ ``there is no $F$-closed class model with a Woodin." Let $M$ be a transitive class model closed under $F$ that satisfies $\textsf{\textsf{\textsf{ZFC}}}^- + $ ``there is no class $F$-closed inner model of a Woodin". Futhermore, assume that $\omega_1 \subseteq M$. Let $\mathcal{P} \in M$ be an $F$-premouse with no definable Woodin. Then
\begin{equation*}
\mathcal{P} \textrm{ is a $F$-mouse } \Leftrightarrow M \vDash \mathcal{P}\textrm{ is a $F$-mouse.}
\end{equation*}
\end{lemma}
For a proof of this in the case $F=\textrm{rud}$, see \cite{Woodinaxiom}. The proof of the lemma is just a trivial modification. Again, since we work below $\Omega$ in $M$, we identify $\Omega$ with OR when applying the above lemma. By a theorem of R. Schindler, translated into our context, $K$ is just a stack of $F$-mice above $\omega_2$ (here $\omega_2^J < \kappa$), we have $\M \triangleleft K^J = K$. But $\lambda = (\kappa^+)^K$ and $\M\triangleleft K|\lambda$. Contradiction.   
\\
\indent Now the rest of the proof is just as in that of Theorem 3.1 in \cite{MaxCM}. Let $E_j$ be the superstrong extender derived from $j$. Since card$(N) = \kappa$ and $\lambda < \kappa^+$, a standard argument (due to Kunen) shows that $F, G \in L^{F^\sharp}[j(A)]$ where
\begin{equation*}
F = E_j \cap ([j(\kappa)]^{<\omega}\times K^{L^{F^\sharp}[A]})
\end{equation*}
and,
\begin{equation*}
G = E_j \cap ([j(\kappa)]^{<\omega}\times N).
\end{equation*}
The key is card$(N)^{L^{F^\sharp}[A]} = \kappa$ and card$(K\cap \powerset(\kappa))^{L^{F^\sharp}[A]} = \kappa$. We show $F\in L^{F^\sharp}[j(A)]$. The proof of $G\in L^{F^\sharp}[j(A)]$ is the same. For $a\in [j(\kappa)]^{<\omega}$, let $\langle B_\alpha \ | \ \alpha < \kappa\rangle \in L^{F^\sharp}[A]$ be an enumeration of $\powerset([\kappa]^{|a|})\cap K^{L^{F^\sharp}[A]} = \powerset([\kappa]^{|a|})\cap K^{L^{F^\sharp}[j(A)]}$ and 
\begin{center}
$E_a = \{ B_\alpha \ | \ \alpha < \kappa \wedge a \in j(B_\alpha) \}$.
\end{center}
Then $E_a\in L^{F^\sharp}[j(A)]$ because $\langle j(B_\alpha) \ | \ \alpha < \kappa \rangle \in L^{F^\sharp}[j(A)]$.

Hence $(K^{L^{F^\sharp}[j(A)]},F)$ and $(N,G)$ are elements of $L^{F^\sharp}[j(A)]$. In $L^{F^\sharp}[j(A)]$, for cofinally many $\xi < j(\kappa)$, $F|\xi$ coheres with $K$ and $(N,G)$ is a weak $\mathcal{A}$-certificate for $(K,F\rest\xi)$ (in the sense of \cite{MaxCM}), where
\begin{equation*}
\mathcal{A} = \bigcup_{n<\omega} \powerset([\kappa]^n)^K.
\end{equation*}
By Theorem 2.3 in \cite{MaxCM}, those segments of $F$ are on the extender sequence of $K^{L^{F^\sharp}[j(A)]}$. But then $\kappa$ is Shelah in $K^{L^{F^\sharp}[j(A)]}$, which is a contradiction.  
\end{proof}
The proof of the claim also shows that $(\kappa^+)^K = (\kappa^+)^J$ for any set (of size smaller than $\Omega$) generic extension $J$ of $H_x$. In particular, since any $A \subseteq \omega^V_1 = \kappa$ belongs to a set generic extension of $H_x$ of size smaller than $\Omega$, we immediately get that $(\kappa^+)^K = \omega_2$. This is impossible in the presence of $\mu$.\footnote{The argument we're about to give is based on Solovay's proof that square fails above a supercompact cardinal.} To see this, let $\vec{C} = \langle C_\alpha \ | \ \alpha < \omega_2\rangle$ be the canonical $\square_{\kappa}$-sequence in $K$. The existence of $\vec{C}$ follows from the proof of the existence of square sequences in pure $L[E]$-models in \cite{schimmerling2004characterization}. Working in $V$, let $\nu$ be the measure on $\powerset_{\omega_1}(\omega_2)$ induced by $\mu$ defined as follows. First, fix a surjection $\pi: \mathbb{R} \rightarrow \omega_2$. Then $\pi$ trivially induces a surjection from $\powerset_{\omega_1}(\mathbb{R})$ onto $\powerset_{\omega_1}(\omega_2)$ which we also call $\pi$. Then our measure $\nu$ is defined as
\begin{equation*}
A \in \nu \Leftrightarrow \pi^{-1}[A] \in \mu.
\end{equation*}
Now consider the ultrapower map $j: K \rightarrow Ult(K,\nu) = K^*$ (where the ultrapower uses all functions in $V$). An easy calculation gives us that $j''\omega_2 = [\lambda\sigma.\sigma]_\nu$ and $A \in \nu \Leftrightarrow j''\omega_2 \in j(A)$. So let $\gamma = \textrm{sup}j''\omega_2$ and $\vec{D} = j(\vec{C}) \in K^*$. Note that $(\kappa^+)^{K^*} = \omega_2$ and since $K^* \vDash \textsf{\textsf{ZF}C}$, $\omega_2$ is regular in $K^*$. Also $\gamma < j(\omega_2^V)$. Now consider the set $D_\gamma$. By definition, $D_\gamma$ is an club in $\gamma$ so it has order type at least $\omega_2$. However, let $C = \langle \alpha < \omega_2 \ | \ \textrm{cof}(\alpha) = \omega\rangle$. Then $j(C) = j''C$ is an $\omega$-club in $\gamma$. Hence $E=D_\gamma \cap j(C)$ is an $\omega$-club in $\gamma$. For each $\alpha \in$ lim$E \subseteq E$, $D_\alpha = D_\gamma \cap \alpha$ and $D_\alpha$ has order type strictly less than $\omega_1$ (this is because cof$(\alpha)=\omega$). This implies that every proper initial segment of $D_\gamma$ has order type strictly less than $\omega_1$ which is a contradiction. 
\end{proof}
The lemma shows that the mouse operator $\M_1^{\sharp,F}$ and hence the corresponding model operator $F_{\M_1^{\sharp,F}}$ is defined for all $x\in \mathbb{R}$ coding $a$. Since $F$, $F^\sharp$ relativize well, $\M_1^{\sharp,F}$ and hence $F_{\M_1^{\sharp,F}}$ are defined on all $H_{\omega_1}$ (in $M$ as well as in $V$) above $a$. This implies that $\forall^*_\mu \sigma$, $F_{\M_1^{\sharp,F}}\rest M_\sigma\in M_\sigma$ . By $\L\acute{o}s$, the $F_{\M_1^{\sharp,F}}$-operator is defined on $H^M_{\Omega}$ (above $a$). Also, for each $y\in H_\Omega^M$, $\M_1^{\sharp,F}(y)$ is $(\omega,\Omega,\Omega+1)$-iterable. This completes the proof of the theorem.
\end{proof}
\begin{theorem}
\label{the cmi theorem}
$M_0 =_{def} K(\mathbb{R}) \vDash \textsf{\textsf{AD}}^+ + \Theta = \theta_0$.
\end{theorem}

The proof of Theorem \ref{the cmi theorem} is very much like the proof of the core model induction theorems in \cite{sargsyanstrength}, \cite{CMI} (see Chapter 7) and \cite{PFA} using the scales analysis developed in \cite{ScalesK(R)} and \cite{Scalesweakgap}. However, there is one point worth going over. 

Suppose we are doing the core model induction to prove \rthm{the cmi theorem}.  During this core model induction, we climb through the levels of $K(\mathbb{R})$ some of which project to $\mathbb{R}$ but do not satisfy that $``\Theta=\theta_0"$. It is then the case that the scales analysis of \cite{ScalesK(R)}, \cite{Scalesendgap} cannot help us in producing the next ``new" set. However, such levels can never be problematic for proving that $\textsf{\textsf{AD}}^+$ holds in $K(\mathbb{R})$. This follows from the following lemma. Again, we remind the reader that the proof of Lemma \ref{od levels are enough} and hence of Theorem \ref{the cmi theorem} makes heavy use of the result proved in Theorem \ref{PD}.

\begin{lemma}\label{od levels are enough} Suppose $\M\lhd K(\mathbb{R})$ is such that $\rho(\M)=\mathbb{R}$ and $\M\models  ``\Theta\not=\theta_{\Sigma}"$. Then there is $\N\lhd K(\mathbb{R})$ such that $\M\lhd \N$, $\N\models ``\textsf{\textsf{AD}}^+ + \Theta=\theta_0"$.
\end{lemma}
\begin{proof} Since $\M\models ``\Theta\not = \theta_0"$ it follows that $\powerset(\mathbb{R})^\M\cap (K(\mathbb{R}))^\M\not= \powerset(\mathbb{R})^\M$ by a result of G. Sargsyan and J. Steel, see \cite{DMATM}. It then follows that there is some $\a<o(\M)$ such that $\rho(\M|\a)=\mathbb{R}$ but $\M|\a\not \insegeq (K(\mathbb{R}))^\M$. Let $\pi: \N\rightarrow \M|\a$ be such that $\N$ is countable and its iteration strategy is not in $\M$. Let $\Lambda$ be the $(\omega,\omega_1,\omega_1)$-iteration strategy of $\N$. Using the measure $\mu$ we can lift $\Lambda$ to an $(\Omega,\Omega)$-strategy (in $M$). Then a core model induction through $L_{\Omega}^\Lambda(\mathbb{R})$ (using Theorem \ref{PD}) shows that $L^\Lambda(\mathbb{R})\footnote{We cut off the model at $\Omega$ and pretend that $\Omega$ is OR.}\models \textsf{\textsf{AD}}^+$ (this is where we needed clause 3 of the previous section) and so $L(\Lambda,\mathbb{R})\vDash \textsf{AD}^+$. Furthermore, it's easy to see that $L(\Lambda,\mathbb{R})\models ``V=L(\powerset(\mathbb{R})) + \Theta=\theta_0"$. It then follows from the aforementioned result of G. Sargsyan and J. Steel that $L(\Lambda,\mathbb{R})\models \powerset(\mathbb{R})=\powerset(\mathbb{R})\cap K(\mathbb{R})$. Let then $\K\insegeq (K(\mathbb{R}))^{L(\Lambda,\mathbb{R})}$ be such that $\rho(\K)=\mathbb{R}$, $\K\models \Theta=\theta_0$ and $\Lambda\in \K$ (there is such a $\K$ by an easy application of $\Sigma^2_1$ reflection in $L(\Lambda,\mathbb{R})$). Since countable submodels of $\K$ are iterable , we have that $\K\insegeq K(\mathbb{R})$. Also we cannot have that $\K \inseg \M$ as otherwise $\N$ would have a strategy in $\M$. Therefore, $\M\insegeq \K$.
\end{proof}

We can now do the core model induction through the levels of $K(\mathbb{R})$ as follows. If we have reached a gap satisfying $``\Theta=\theta_0"$ then we can use the scales analysis of \cite{ScalesK(R)} and \cite{Scalesendgap} to go beyond. If we have reached a level that satisfies $``\Theta\not=\theta_0"$ then using \rlem{od levels are enough} we can skip through it and go to the least level beyond it that satisfies $``\Theta=\theta_0"$. We leave the rest of the details to the reader. This completes our proof sketch of Theorem \ref{the cmi theorem}.
\\ 
\\
\noindent \textbf{Remark.} $\textsf{\textsf{AD}}^{K(\mathbb{R})}$ is the most amount of determinacy one could hope to prove. This is because if $\mu$ comes from the Solovay measure (derived from winning strategies of real games) in an $\textsf{\textsf{AD}}^+ + \textsf{\textsf{AD}}_{\mathbb{R}} + \rm{\textsf{SMSC}}$ universe, call it $V$ (any $\textsf{\textsf{AD}}_{\mathbb{R}} + V=L(\powerset(\mathbb{R}))$-model below ``$\textsf{\textsf{AD}}_{\mathbb{R}}+\Theta$ is regular" would do here), then $L(\mathbb{R},\mu)^V \cap \powerset(\mathbb{R}) \subseteq K(\mathbb{R})^V$. This is because $\mu$ is $\mathrm{OD}$ hence $\powerset{(\mathbb{R})} \cap L(\mu,\mathbb{R}) \subseteq \powerset_{\theta_0}(\mathbb{R})$. Since $\textsf{AD}^+ +\textsf{S\textsf{MC}}$ gives us that any set of reals of Wadge rank $< \theta_0$ is contained in an $\mathbb{R}$-mouse (by an unpublished result of Sargsyan and Steel but see \cite{DMATM}), we get that $\powerset{(\mathbb{R})} \cap L(\mu,\mathbb{R}) \subseteq K(\mathbb{R})$ (it is conceivable that the inclusion is strict). By Theorem \ref{mu unique}, $L(\mathbb{R},\mu) \vDash \Theta = \theta_0$, which implies $L(\mathbb{R},\mu)\vDash \powerset(\mathbb{R}) \subseteq K(\mathbb{R})$. Putting all of this together, we get $L(\mathbb{R},\mu) \vDash K(\mathbb{R}) = L(\powerset{(\mathbb{R})}) + \textsf{\textsf{AD}}^+$.
\\
\indent The above remark suggests that we should try to show that every set of reals in $V = L(\mathbb{R},\mu)$ is captured by an $\mathbb{R}$-mouse, which will prove Theorem \ref{characterization of AD}. This is accomplished in the next three sections.

\section{\textsf{AD} in $L(\mathbb{R},\mu)$}\label{getBranchCondensation}
First we show $\Theta = \Theta^{K(\mathbb{R})}$. Suppose for contradiction that $\Theta^{K(\mathbb{R})} < \Theta$. We first show that there is a model containing $\mathbb{R}\cup \mathrm{OR}\cup K(\mathbb{R})$ that satisfies $\textsf{AD}^++\Theta>\theta_0$". The argument will closely follow the argument in Chapter 7 of \cite{CMI}. All of our key notions and notations come from there unless specified otherwise. Let $\Theta^* = \Theta^{K(\mathbb{R})}$. Let $\M_\infty$ be $\rm{\rm{H\mathrm{OD}}}^{K(\mathbb{R})}\rest\Theta^*$. Then $\M_\infty = \M_\infty^+\rest\Theta^*$ where $\M_\infty^+$ is the limit of a directed system (the hod limit system) indexed by pairs $(\P,\vec{A})$ where $\P$ is a suitable premouse, $\vec{A}$ is a finite sequence of $\mathrm{OD}$ sets of reals, and $\P$ is strongly $\vec{A}$-quasi-iterable in $K(\mathbb{R})$. For more details on how the direct limit system is defined, the reader should consult Chapter 7 of \cite{CMI}. Let $\Gamma$ be the collection of $\mathrm{OD}^{K(\mathbb{R})}$ sets of reals. For each $\sigma \in \powerset_{\omega_1}(\mathbb{R})$ such that $Lp(\sigma) \vDash \textsf{\textsf{AD}}^+$, let $\M_\infty^\sigma$ and $\Gamma^\sigma$ be defined the same as $\M_\infty$ and $\Gamma$ but in $Lp(\sigma)$. Let $\Theta^\sigma = o(\M_\infty^\sigma)$. By $\textsf{\textsf{AD}}^{K(\mathbb{R})}$ and $\Theta^* < \Theta$, we easily get
\begin{lemma}
\label{uncollapse}
$\forall_{\mu}^{*} \sigma (Lp(\sigma)\vDash \textsf{\textsf{AD}}^+, \ and \ there \ is \ an \ elementary \ map \ \pi_\sigma: (Lp(\sigma), \M_\infty^\sigma,\Gamma^\sigma) \rightarrow (K(\mathbb{R}), \M_\infty,\Gamma).)$
\end{lemma}
\begin{proof}
First, it's easily seen that $K(\mathbb{R}) \vDash \textsf{\textsf{AD}}^+$ implies $\forall_{\mu}^{*} \sigma Lp(\sigma)\vDash \textsf{\textsf{AD}}^+$. We also have that letting $\nu$ be the induced measure on $\powerset_{\omega_1}(K(\mathbb{R}))$
\begin{equation*}
\forall^*_\nu X \ X \prec K(\mathbb{R}).
\end{equation*}
The second clause of the lemma follows by transitive collapsing the $X$'s above. Note that $\forall^*_\mu \sigma \ Lp(\sigma)$ is the uncollapse of some countable $X \prec K(\mathbb{R})$ such that $\mathbb{R}^X = \sigma$. This is because if $\M$ is an $\mathbb{R}$-mouse then $\forall^*_\nu X$ $\M \in X$. The $\pi_\sigma$'s are just the uncollapse maps.
\end{proof}
We may as well assume $(\forall^*_\mu \sigma)(Lp(\sigma) = Lp(\sigma)^{K(\mathbb{R})})$ as otherwise, fix a $\sigma$ such that $Lp(\sigma) \vDash \textsf{\textsf{AD}}^+$ and $\M \lhd$ $Lp(\sigma)$ a sound mouse over $\sigma$, $\rho_\omega(\M) = \sigma$ and $\M \notin$ $Lp(\sigma)^{K(\mathbb{R})}$. Let $\Lambda$ be the strategy of $\M$. Then by a core model induction as above, we can show that $L^\Lambda(\mathbb{R}) \vDash \textsf{\textsf{AD}}^+ + \Theta > \theta_0$. Since this is very similar to the proof of $PD$, we only mention a few key points for this induction. First, $\Lambda$ is a $\omega_1+1$ strategy with condensation and $\forall^*_\mu \sigma$ $\Lambda\rest M_\sigma \in M_\sigma$ and $\forall^*_\mu \sigma$ $\Lambda\rest H_\sigma[\M] \in H_\sigma[\M]$. This allows us to lift $\Lambda$ to a $\Omega+1$ strategy in $M$ and construct $K^\Lambda$ up to $\Omega$ inside $\prod_\sigma H_\sigma[\M]$. This is a contradiction to our smallness assumption.
\begin{lemma}
\label{full}
$\forall_{\mu}^{*} \sigma$ \ $\M_\infty^\sigma$ is full in $K(\mathbb{R})$ in the sense that $Lp(\M_\infty^\sigma) \vDash \Theta^\sigma$ is Woodin.
\end{lemma}
\begin{proof}
First note that $Lp_2(\sigma) =_{def} Lp(Lp(\sigma)) \vDash \textsf{AD}^+ + \Theta = \theta_0$ because $\powerset(\mathbb{R})^{Lp_2(\sigma)}= \powerset(\mathbb{R})^{Lp(\sigma)}$. So suppose $\N^\sigma \rhd \M^\sigma_\infty$ is the $Q$-structure. It's easy to see that $\N^\sigma \in Lp_2(\sigma)$ and is in fact $\mathrm{OD}$ there.
\\
\indent Next we observe that in $Lp_2(\sigma)$, $\Theta = \Theta^\sigma$. By a Theorem of Woodin, we know $\rm{H\mathrm{OD}}^{Lp_2(\sigma)} \vDash \Theta^\sigma$ is Woodin (see Theorem 5.6 of \cite{koellner2010large}). But this is a contradiction to our assumption that $\N^\sigma$ is a $Q$-structure for $\Theta^\sigma$. 
\end{proof}
The last lemma shows that for a typical $\sigma$, $Lp_\omega(\M^\sigma_\infty)$ is suitable in $K(\mathbb{R})$. Let $\M_\infty^{\sigma,+}$ be the hod limit computed in $Lp(\sigma)$. Let $(\Gamma^\sigma)^{<\omega} = \{\vec{A_n} \ | \ n < \omega\}$ and for each $n<\omega$, let $\N_n$ be such that $\N_n$ is strongly $\vec{A_n}$-quasi-iterable in $Lp(\sigma)$ such that $\M_\infty^{\sigma,+}$ is the quasi-limit of the $\N_n$'s in $Lp(\sigma)$. Let $\M_\infty^{\sigma,*}$ be the quasi-limit of the $\N_n$'s in $K(\mathbb{R})$. We'll show that $\pi_\sigma''\Gamma^\sigma$ is cofinal in $\Gamma$, $\M_\infty^{\sigma,+} = \M_\infty^{\sigma,*} = Lp_\omega(\M^\sigma_\infty)$ and hence $\M_\infty^{\sigma,+}$ is strongly $A$-quasi-iterable in $K(\mathbb{R})$ for each $A \in \pi_\sigma''\Gamma^\sigma$. From this we'll get a strategy $\Sigma_\sigma$ for $\M_\infty^{\sigma,+}$ with weak condensation. This proceeds much like the proof in Chapter 7 of \cite{CMI}.
\\
\indent Let $T$ be the tree for a universal $(\Sigma^2_1)^{K(\mathbb{R})}$-set; let $T^* = \prod_\sigma T\slash \mu$ and $T^{**} = \prod_\sigma T^*\slash \mu$. To show $(\forall_{\mu}^{*} \sigma)(\pi_\sigma''\Gamma^\sigma$ is cofinal in $\Gamma$) we first observe that 
\begin{equation*}
(\forall^*_\mu \sigma)(L[T^*,\M_\infty^\sigma]|\Theta^\sigma = \M^\sigma_\infty),
\end{equation*}
that is, $T^*$ does not create $Q$-structures for $\M_\infty^\sigma$. This is because $\M^\sigma_\infty$ is countable, $\omega_1^V$ is inaccessible in any inner model of choice, $L[T^*,\M_\infty^\sigma]|\omega_1^V = L[T,\M_\infty^\sigma]|\omega_1^V$, and $L[T,\M^\sigma_\infty]|\Theta^\sigma = \M^\sigma_\infty$ by Lemma \ref{full}. Next, let $E_\sigma$ be the extender derived from $\pi_\sigma$ with generators in $[\gamma]^{<\omega}$, where $\gamma$ = sup$\pi_\sigma''\Theta^\sigma$. By the above, $E_\sigma$ is a pre-extender over $L[T^*,\M_\infty^\sigma]$.
\begin{lemma}
\label{wellfounded}
$(\forall^*_\sigma \sigma)(Ult(L[T^*,\M_\infty^\sigma],E_\sigma) \ is \ wellfounded)$.
\end{lemma}
\begin{proof}
The statement of the lemma is equivalent to
\begin{equation*}
Ult(L[T^{**},\M_\infty],\Pi_\sigma E_\sigma\slash \mu) \textrm{ is wellfounded}. \ \ (\textasteriskcentered)
\end{equation*}
To see (\textasteriskcentered), note that 
\begin{equation*}
\prod_\sigma E_\sigma = E_\mu
\end{equation*}
where $E_\mu$ is the extender from the ultrapower map $j_\mu$ by $\mu$  (with generators in $[\xi]^{<\omega}$, where $\xi$ = sup$j_\mu''\Theta^*$). This uses normality of $\mu$. We should metion that the equality above should be interpreted as saying: the embedding by $\Pi_\sigma E_\sigma\slash \mu$ agrees with $j_\mu$ on all ordinals (less than $\Theta$).
\\
\indent Since $\mu$ is countably complete and \textsf{DC} holds, we have that $Ult(L[T^{**},\M_\infty],E_\mu)$ is wellfounded. Hence we're done. 
\end{proof}
\begin{theorem}
\label{weak condensation}
\begin{enumerate}
\item $(\forall^*_\mu \sigma)(\pi_\sigma \textrm{ is continuous at } \theta^\sigma)$. Hence \ $cof(\Theta^{K(\mathbb{R})}) = \omega$.
\item If \ $i: \M_\infty^\sigma \rightarrow S$, and $j: S \rightarrow \M_\infty$ are elementary and $\pi_\sigma = j \circ i$ and S is countable in $K(\mathbb{R})$, then $S$ is full in $K(\mathbb{R})$. In fact, if $W$ is the collapse of a hull of $S$ containing $rng(i)$, then $W$ is full in $K(\mathbb{R})$.
\end{enumerate}
\begin{proof}
The keys are Lemma \ref{wellfounded} and the fact that the tree $T^*$, which enforces fullness for $\mathbb{R}$-mice, does not generate $Q$-structures for $\M^\sigma_\infty$. To see (1), suppose not. Fix a typical $\sigma$ for which (1) fails. Let $\gamma$ = sup$\pi_\sigma''\Theta^\sigma < \Theta^*$. Let $E_\sigma$ be the extender derived from $\pi_\sigma$ with generators in $[\gamma]^{<\omega}$ and consider the ultrapower map
\begin{equation*}
\tau: L[T^{*},\M_\infty^\sigma] \rightarrow N_\sigma =_{def} Ult(L[T^{*},\M_\infty^\sigma],E_\sigma).
\end{equation*}
We may as well assume $N_\sigma$ is transitive by Lemma \ref{wellfounded}. We have that $\tau$ is continuous at $\Theta^\sigma$ and $N_\sigma \vDash o(\tau(\M_\infty^\sigma))$ is Woodin. Since $o(\tau(\M_\infty^\sigma)) = \gamma < \Theta^*$, there is a $Q$-structure $\Q$ for $o(\tau(\M_\infty^\sigma))$ in $K(\mathbb{R})$. But $\Q$ can be constructed from $T^{*}$, hence from $\tau(T^{*})$. To see this, suppose $\Q = \Pi_\sigma \Q_\sigma\slash \mu$ and $\gamma = \Pi_\sigma \gamma_\sigma\slash \mu$. Then $\forall^*_\mu \sigma$ $\Q_\sigma$ is the $Q$-structure for $\M^\sigma_\infty|\gamma_\sigma$ and the iterability of $\Q_\sigma$ is certified by $T$. This implies the iterability of $\Q$ is certified by $T^*$. But $\tau(T^{*}) \in N_\sigma$, which does not have $Q$-structures for $\tau(\M_\infty^\sigma)$. Contradiction.
\\
\indent (1) shows then that $\pi_\sigma''\Gamma^\sigma$ is cofinal in $\Gamma$. The proof of (2) is similar. We just prove the first statement of (2). The point is that $i$ can be lifted to an elementary map
\begin{equation*}
i^*: L[T^{*}, \M_\infty^\sigma] \rightarrow L[\overline{T}, S]
\end{equation*}
for some $\overline{T}$ and $j$ can be lifted to 
\begin{equation*}
j^*: L[\overline{T}, S] \rightarrow N_\sigma
\end{equation*}
by the following definition
\begin{equation*}
j^*(i^*(f)(a)) = \tau(f)(j(a))
\end{equation*}
for $f \in L[T^*,\M^\sigma_\infty]$ and $a \in [o(\S)]^{<\omega}$. By the same argument as above, $\overline{T}$ certifies iterability of mice in $K(\mathbb{R})$ and hence enforces fullness for $S$ in $K(\mathbb{R})$. This is what we want.
\end{proof}
\end{theorem}
We can define a map $\tau: \M_\infty^{\sigma,+}\rightarrow \M_\infty^{\sigma,*}$ as follows. Let $x \in \M_\infty^{\sigma,+}$. There is an $i<\omega$ and a $y$ such that in $Lp(\sigma)$, $x = \pi^{A_i}_{\N_i,\infty}(y)$, where $\pi^{A_i}_{\N_i,\infty}$ is the direct limit map from $H^{\N_i}_{A_i}$ into $\M_\infty^{\sigma,+}$ in $Lp(\sigma)$. Let
\begin{equation*}
\tau(x) = \pi^{A_i}_{\N_i,\M_\infty^{\sigma,*}}(y),
\end{equation*}
where $\pi^{A_i}_{\N_i,\M_\infty^{\sigma,*}}$ witnesses $(\N_i,A_i) \preceq (\M_\infty^{\sigma,*},A_i)$ in the hod direct limit system in $K(\mathbb{R})$.
\begin{lemma}
\label{equality}
\begin{enumerate}
\item $\M_\infty^{\sigma,*} = H^{\M_\infty^{\sigma,*}}_{\pi_\sigma''\Gamma_\sigma}$; furthermore, for any quasi-iterate $\Q$ of $\M_\infty^{\sigma,*}$, $\Q = H^{\Q}_{\pi_\sigma''\Gamma_\sigma}$ and $\pi^{\pi_\sigma''\Gamma_\sigma}_{\M_\infty^{\sigma,*},\Q}(\tau^{\M_\infty^{\sigma,+}}_A)= \tau^{\Q}_A$ for all $A \in \pi_\sigma''\Gamma_\sigma$.
\item $\tau = id$ and $\M_\infty^{\sigma,+} = \M_\infty^{\sigma,*}$.
\item $\pi_\sigma = \pi^{\pi_\sigma''\Gamma_\sigma}_{\M_\infty^{\sigma,+},\infty}$.
\end{enumerate}
\end{lemma}
\begin{proof}
The proof is just that of Lemmata 7.8.7 and 7.8.8 in \cite{CMI}. We first show (1). In this proof, ``suitable" means suitable in $K(\mathbb{R})$. The key is for any quasi-iterate $\Q$ of $\M_\infty^{\sigma,*}$, we have
\begin{equation*}
\pi_\sigma|\M_\infty^{\sigma,+} = \pi^{\pi_\sigma''\Gamma_\sigma}_{\Q,\infty}
\circ\pi^{\pi_\sigma''\Gamma_\sigma}_{\M_\infty^{\sigma,*},\Q}\circ \tau. \ \ \ \ (\textasteriskcentered)
\end{equation*}
Using this and Theorem \ref{weak condensation}, we get $H^{\Q}_{\pi_\sigma''\Gamma_\sigma} = \Q$ for any quasi-iterate $\Q$ of $\M_\infty^{\sigma,*}$. To see this, first note that $\Q$ is suitable; Theorem \ref{weak condensation} implies the collapse $\S$ of $H^{\Q}_{\pi_\sigma''\Gamma_\sigma}$ must be suitable. This means, letting $\delta$ be the Woodin of $\Q$, $H^{\Q}_{\pi_\sigma''\Gamma_\sigma}|(\delta+1) = \Q|(\delta+1)$. Next, we show $H^{\Q}_{\pi_\sigma''\Gamma_\sigma}|((\delta^+)^\Q) = \Q|((\delta^+)^\Q)$. The proof of this is essentially that of Lemma 4.35 in \cite{ketchersid2000toward}. We sketch the proof here. Suppose not. Let $\pi: \S \rightarrow \Q$ be the uncollapse map. Note that crt$(\pi) = (\delta^+)^\S$ and $\pi((\delta^+)^\S) = (\delta^+)^\Q$. Let $\R$ be the result of first moving the least measurable of $\Q|((\delta^+)^\Q)$ above $\delta$ and then doing the genericity iteration (inside $\Q$) of the resulting model to make $\Q|\delta$ generic at the Woodin of $\R$. Let $\mathcal{T}$ be the resulting tree. Then $\mathcal{T}$ is maximal with $lh(\mathcal{T}) = (\delta^+)^\Q$; $\R = Lp(\M(\mathcal{T}))$; and the Woodin of $\R$ is $(\delta^+)^\Q$. Since $\{\gamma^\R_A \ | \ A \in \pi_\sigma''\Gamma_\sigma\}$ are definable from $\{\tau^\Q_{A,(\delta^+)^\Q} \ | \ A \in \pi_\sigma''\Gamma_\sigma\}$, they are in rng$(\pi)$. This gives us that sup$H^{\Q}_{\pi_\sigma''\Gamma_\sigma}\cap (\delta^+)^\Q = (\delta^+)^\Q$, which easily implies $(\delta^+)^\Q \subseteq H^{\Q}_{\pi_\sigma''\Gamma_\sigma}$. The proof that $(\delta^{+n})^\Q\subseteq H^{\Q}_{\pi_\sigma''\Gamma_\sigma}$ for $1 < n < \omega$ is similar and is left for the reader.
\\
\indent (2) easily follows from (1). (3) follows using (\textasteriskcentered) and $\tau = id$.
\end{proof}
\indent For each $\sigma$ such that Theorem \ref{weak condensation} and Lemma \ref{equality} hold for $\sigma$, let $\Sigma_\sigma$ be the canonical strategy for $\M_\infty^\sigma$ as guided by $\pi_\sigma''\Gamma^\sigma$. Recall $\pi_\sigma''\Gamma^\sigma$ is a cofinal collection of $\mathrm{OD}^{K(\mathbb{R})}$ sets of reals. The existence of $\Sigma_\sigma$ follows from Theorem 7.8.9 in \cite{CMI}. Note that $\Sigma_\sigma$ has weak condensation, i.e., suppose $\Q$ is a $\Sigma_\sigma$ iterate of $\M_\infty^{\sigma,+}$ and $i:\M_\infty^{\sigma,+} \rightarrow \Q$ is the iteration map, and suppose $j:\M_\infty^{\sigma,+} \rightarrow \R$ and $k:\R \rightarrow \Q$ are such that $i = k\circ j$ then $\R$ is suitable (in the sense of $K(\mathbb{R})$). 
\begin{definition}[Branch condensation]
\label{branch condensation}
Let $\M_\infty^{\sigma,+}$ and $\Sigma_\sigma$ be as above. We say that $\Sigma_\sigma$ has branch condensation if for any $\Sigma_\sigma$ iterate $\Q$ of $\M_\infty^{\sigma,+}$, letting $k: \M_\infty^{\sigma,+} \rightarrow \Q$ be the iteration map, for any maximal tree $\mathcal{T}$ on $\M_\infty^{\sigma,+}$, for any cofinal non-dropping branch $b$ of $\mathcal{T}$, letting $i = i^{\mathcal{T}}_b$, $j: \M^{\mathcal{T}}_b \rightarrow \P$, where $\P$ is a $\Sigma_\sigma$ iterate of $\M_\infty^\sigma$ with iteration embedding $k$, suppose $k = j\circ i$, then $b = \Sigma_\sigma(\mathcal{T})$. 
\end{definition}
\begin{theorem}
\label{branch condensation}
$(\forall^*_\mu \sigma)(\textrm{A tail of }\Sigma_\sigma \textrm{ has branch condensation.})$
\end{theorem}
\begin{proof}
The proof is like that of Theorem 7.9.1 in \cite{CMI}. We only mention the key points here. We assume that $\forall^*_\mu \sigma$ no $\Sigma_\sigma$-tails have branch condensation. Fix such a $\sigma$. First, let $X_\sigma =$ rng$(\pi_\sigma\rest \M_\infty^{\sigma,+})$ and 
\begin{equation*}
H = \rm{\rm{H\mathrm{OD}}}_{\{\mu, \M_\infty^{\sigma,+}, \M_\infty, \pi_\sigma, T^*, X_\sigma, x_\sigma\}},
\end{equation*}
where $x_\sigma$ is a real enumerating $M_\infty^{\sigma,+}$. So $H \vDash \textsf{\textsf{\textsf{ZF}C}} + ``\M_\infty^\sigma$ is countable and $\omega_1^V$ is measurable."
\\
\indent Next, let $\overline{H}$ be a collapse of a countable elementary substructure of a sufficiently large rank-initial segment of $H$. Let $(\gamma,\rho,\N,\nu)$ be the preimage of $(\omega_1^V,\pi_\sigma, \M_\infty,\mu)$ under the uncollapse map, call it $\pi$. We have that $\overline{H} \vDash \textsf{\textsf{\textsf{ZF}C}}^- + $ ``$\gamma$ is a measurable cardinal as witnessed by $\nu$." This $\overline{H}$ will replace the countable iterable structure obtained from the hypothesis HI(c) in Chapter 7 of \cite{CMI}. Now, in $K(\mathbb{R})$, the following hold true:
\begin{enumerate}
\item There is a term $\tau \in \overline{H}$ such that whenever $g$ is a generic over $\overline{H}$ for $Col(\omega, <\gamma)$, then $\tau^g$ is a $(\rho,\M_\infty^{\sigma,+},\N)$-certified bad sequence. See Definitions 7.9.3 and 7.9.4 in \cite{CMI} for the notions of a bad sequence and a $(\rho,\M_\infty^{\sigma,+},\N)$-certified bad sequence respectively.
\item Whenever $i: \overline{H} \rightarrow J$ is a countable linear iteration map by the measure $\nu$ and $g$ is $J$-generic for $Col(\omega, < i(\gamma))$, then $i(\tau)^g$ is truly a bad sequence.
\end{enumerate}
The proof of (1) and (2) is just like that of Lemma 7.9.7 in \cite{CMI}. The key is that in (1), any $(\rho,\M_\infty^{\sigma,+},\N)$-certified bad sequence is truly a bad sequence from the point of view of $K(\mathbb{R})$ and in (2), any countable linear iterate $J$ of $\overline{H}$ can be realized back into $H$ by a map $\psi$ in such a way that $\pi = \psi\circ i$.
\\
\indent Finally, using (1), (2), the iterability of $\overline{H}$, and an $\textsf{\textsf{AD}}^+$-reflection in $K(\mathbb{R})$ like that in Theorem 7.9.1 in \cite{CMI}, we get a contradiction.
\end{proof}
Fix a $\sigma$ as in Theorem \ref{branch condensation} and let $\Sigma'_\sigma$ be a tail of $\Sigma_\sigma$ with branch condensation. By the construction of $\Sigma_\sigma$, $\Sigma'_\sigma\notin K(\mathbb{R})$. Using $\mu$ and the fact that $\forall^*_\mu \tau \ \Sigma'_\sigma\rest M_\tau \in M_\tau$, we can lift $\Sigma'_\sigma$ to a strategy on $H_\Omega^M$ with branch condensation, which we also call $\Sigma'_\sigma$. An argument as in Lemma \ref{od levels are enough} shows that $K^{\Sigma'_\sigma}(\mathbb{R})\footnote{$K^{\Sigma'_\sigma}(\mathbb{R})$ is the union of all sound $\Sigma'_\sigma$-$\mathbb{R}$-premice $\M$ such that $\rho_{\omega}(\M)=\mathbb{R}$ and every countable $\M^*$ embeddable into $\M$ has a unique $(\omega,\omega_1+1)$-iteration $\Sigma'_\sigma$ strategy. See \cite{trang2012scales} or \cite{trang2013generalized} for a constructive definition of $K^{\Sigma'_\sigma}(\mathbb{R})$.} \vDash \textsf{AD}^+$. This along with the fact that $\Sigma'_\sigma \notin K(\mathbb{R})$ imply
\begin{equation*}
K^{\Sigma'_\sigma}(\mathbb{R}) \vDash \Theta > \theta_0.
\end{equation*}
Furthermore, since $\Sigma'_\sigma$ is $K(\mathbb{R})$-fullness preserving (in fact is guided by a self-justifying system cofinal in $\powerset(\mathbb{R})\cap K(\mathbb{R})$), $\powerset(\mathbb{R})\cap K(\mathbb{R})$ is a Wadge initial segment of $\powerset(\mathbb{R})\cap L(\Sigma'_\sigma,\mathbb{R})$.

Recall that $T$ is the tree for a universal $\Sigma^2_1$-in-$K(\mathbb{R})$-set and $T^* = \prod_\sigma T\slash \mu$.
\begin{lemma}
\label{fullness}
Let $\powerset(\mathbb{R}) \cap L(T^*,\mathbb{R}) = \powerset(\mathbb{R}) \cap K(\mathbb{R})$.
\end{lemma}
\begin{proof}
By \textsf{MC} in $K(\mathbb{R})$, we have 
\begin{equation*}
(\forall^*_\mu \sigma) (\powerset(\sigma) \cap L(T,\sigma) = Lp(\sigma)\cap \powerset(\sigma)).
\end{equation*}
This proves the lemma.
\end{proof}
We now show that $\mu$ is amenable to $K(\mathbb{R})$ in the sense that $\mu$ restricting to any Wadge initial segment of $\powerset(\mathbb{R})^{K(\mathbb{R})}$ is in $K(\mathbb{R})$.
\begin{lemma}
\label{amenability}
Suppose $S = \{(x,A_x) \ | \ x \in \mathbb{R} \ \wedge \ A_x \in \powerset(\powerset_{\omega_1}(\mathbb{R})) \} \in K(\mathbb{R})$. Then $\mu \rest S = \{ (x,A_x) \ | \ \mu(A_x) = 1\} \in K(\mathbb{R})$.
\end{lemma}
\begin{proof}
Let $A_S$ be an $\infty$-Borel code\footnote{If $S\subseteq \mathbb{R}$, $A_S$ is an $\infty$-Borel code for $S$ if $A_S = (T,\psi)$ where $T$ is a set of ordinals and $\psi$ is a formula such that for all $x\in\mathbb{R}$, $x\in S \Leftrightarrow L[T,x] \vDash\psi[T,x]$.} for $S$ in $K(\mathbb{R})$. We may pick $A_S$ such that it is a bounded subset of $\Theta^*$. We may as well assume that $A_S$ is $\mathrm{OD}^{K(\mathbb{R})}$ and $A_S$ codes $T$. This gives us
\begin{equation*}
(\forall^*_\mu \sigma) (\powerset(\sigma) \cap L(A_S,\sigma) = \powerset(\sigma)\cap L(T,\sigma)),
\end{equation*}
or equivalently letting $A_S^* = \prod_\sigma A_S$,
\begin{equation*}
\powerset(\mathbb{R}) \cap L(A_S^*,\mathbb{R}) =  L(T^*,\mathbb{R}).
\end{equation*}
We have the following equivalences:
\begin{eqnarray*}
(x,A_x) \in \mu\rest S &\Leftrightarrow& (\forall^*_\mu \sigma)(\sigma \in A_x\cap \powerset_{\omega_1}(\sigma)) \\ &\Leftrightarrow& (\forall^*_\mu \sigma)(L(A_S,\sigma) \vDash \emptyset \Vdash_{Col(\omega,\sigma)}\check{\sigma} \in A_x\cap \powerset_{\omega_1}(\sigma)) \\ &\Leftrightarrow& L(A_S^*,\mathbb{R})\vDash \emptyset \Vdash_{Col(\omega,\mathbb{R})}\check{\mathbb{R}} \in A_x.
\end{eqnarray*}
The above equivalences show that $\mu \rest S \in L(S^*,\mathbb{R})$. But by Lemma \ref{fullness} and the fact that $\mu \rest S$ can be coded as a set of reals in $ L(S^*,\mathbb{R})$, hence $\mu\rest S \in L(T^*,\mathbb{R})$, we have that $\mu \rest S \in K(\mathbb{R})$.
\end{proof}
\indent Recall we're trying to get a contradiction from assuming $\Theta>\Theta^{K(\mathbb{R})} = \theta_0^{L(\Sigma'_\sigma,\mathbb{R})}$. Let $M = K^{\Sigma'_\sigma}(\mathbb{R})$ and $H = \rm{\rm{H\mathrm{OD}}}^M_\mathbb{R}$. Note that $\powerset(\mathbb{R})^H = \powerset(\mathbb{R})^{K(\mathbb{R})} = \powerset_{\theta_0}(\mathbb{R})^M$. We aim to show that $L(\mathbb{R},\mu)\subseteq H$, which is a contradiction. By a similar argument as the proof of Theorem \ref{amenability} but relativized to $\Sigma'_\sigma$, $\nu =_{\textrm{def}}\mu\rest \powerset(\mathbb{R})^H \in M$; in fact, letting $\rho$ be the restriction of $\mu$ on the Suslin co-Suslin sets of $M$, then $\rho\in M$. We show $\nu$ is OD in $M$. Let $\pi: \mathbb{R}^\omega \rightarrow \powerset_{\omega_1}(\mathbb{R})$ be the canonical map, i.e. $\pi(\vec{x}) = $rng($\vec{x}$). Let $A \subseteq \powerset_{\omega_1}(\mathbb{R})$ be in $H$. There is a natural interpretation of $A$ as a set of Wadge rank less than $\theta_0^M$, that is the preimage $\vec{A}$ of $A$ under $\pi$ has Wadge rank less than $\theta_0^M$. Fix such an $A$; note that $\vec{A}$ is invariant in the sense that whenever $\vec{x}\in\vec{A}$ and $\vec{y}\in\mathbb{R}^\omega$ and rng$(\vec{x})$ = rng$(\vec{y})$ then $\vec{y}\in\vec{A}$. Let $G_{\vec{A}}$ and $G_A$ be the Solovay games corresponding to $\vec{A}$ and $A$ respectively. In these games, players take turns and play finite sequences of reals and suppose $\langle x_i \ | \ i<\omega\rangle \in \mathbb{R}^\omega$ is the natural enumeration of the reals played in a typical play in either game, then the payoff is as follows:
\begin{equation*}
\textrm{Player I wins the play in } G_{\vec{A}} \textrm{ if } \langle x_i \ | \ i<\omega\rangle \in \vec{A},
\end{equation*}
and
\begin{equation*}
\textrm{Player I wins the play in } G_A \textrm{ if } \{ x_i \ | \ i<\omega\}  \in A.
\end{equation*}
\begin{lemma}
$G_A$ is determined.
\end{lemma}
\begin{proof}
For each $\vec{x} \in \mathbb{R}^\omega$, let $\sigma_{\vec{x}} = $ rng$(\vec{x})$. Consider the games $G^{\vec{x}}_{\vec{A}}$ and $G^{\sigma_{\vec{x}}}_{A}$ which have the same rules and payoffs as those of $ G_{\vec{A}}$ and $G_A$ respectively except that players are required to play reals in $\sigma_{\vec{x}}$. Note that these games are determined and Player I wins the game $G^{\vec{x}}_{\vec{A}}$ iff Player I wins the corresponding game $G^{\sigma_{\vec{x}}}_{A}$.
\\
\indent Without loss of generality, suppose $\nu(\{\sigma\in\powerset_{\omega_1}(\mathbb{R}) \ | \ \textrm{Player I wins } G_A^\sigma\}) = 1$. For each such $\sigma$, let $\tau_\sigma$ be the canonical winning strategy for Player I given by the Moschovakis's Third Periodicity Theorem. We can easily integrate these strategies to construct a strategy $\tau$ for Player I in $G_A$. We know
\begin{equation*}
\forall^*_\rho \sigma \ \tau_\sigma(\emptyset)\in\sigma.
\end{equation*}
We have to use $\rho$ since the set displayed above in general does not have Wadge rank less than $\theta_0$ in $M$. Normality of $\rho$ implies
\begin{equation*}
\exists x\in\mathbb{R} \ \forall^*_\rho \sigma \ \tau_\sigma(\emptyset) = x.
\end{equation*}
Let $\tau(\emptyset) = x$ where $x$ is as above. Now let $y$ be II's response in $G_A$ and $G_{\vec{A}}$. Since $\forall^*_\rho \sigma \ y\in \sigma$, by normality, 
\begin{center}
$\exists z\in \mathbb{R} \forall^*_\rho \sigma \ \tau_\sigma(x,y) = z$.
\end{center}
Let then $\tau(x,y)=z$. It's clear that the above procedure defines $\tau$ on all finite moves. It's easy to show $\tau$ is a winning strategy for Player I in $G_A$.  
\end{proof}
The lemma and standard results of Woodin (see \cite{woodin1983ad}) show that $\rho$ (as defined in the previous lemma) is the unique normal fine measure on the Suslin co-Suslin sets of $M$ and hence $\rho \in \mathrm{OD}^M$. This means $\rho\rest\powerset(\mathbb{R})^H = \nu$ is $\mathrm{OD}$ in $M$. This implies $L(\mathbb{R},\nu) \subseteq H$. But $L(\mathbb{R},\nu) = L(\mathbb{R},\mu)$. Contradiction.
Now we know $\Theta^{K(\mathbb{R})} = \Theta$. We want to show $\powerset(\mathbb{R})\cap K(\mathbb{R}) = \powerset(\mathbb{R})$.
\begin{lemma}
\label{full determinacy}
$\powerset(\mathbb{R}) \cap K(\mathbb{R}) = \powerset(\mathbb{R})$. Hence $L(\mathbb{R},\mu) \vDash \textsf{\textsf{AD}}$.
\end{lemma}
\begin{proof}
First we observe that if $\alpha$ is such that there is a new set of reals in $L_{\alpha+1}(\mathbb{R})[\mu] \backslash L_{\alpha}(\mathbb{R})[\mu]$ then there is a surjection from $\mathbb{R}$ onto $L_{\alpha}(\mathbb{R})[\mu]$. This is because the predicate $\mu$ is a predicate for a subset of $\powerset(\mathbb{R})$, which collapses to itself under collapsing of hulls of $L_{\alpha}(\mathbb{R})[\mu]$ that contain all reals. With this observation, the usual proof of condensation (for $L$) goes through with one modification: one must put all reals into hulls one takes.
\\
\indent Now suppose for a contradiction that there is an $A \in \powerset(\mathbb{R})\cap L(\mathbb{R},\mu)$ such that $A \notin K(\mathbb{R})$. Let $\alpha$ be least such that $A \in L_{\alpha+1}(\mathbb{R})[\mu]\backslash L_{\alpha}(\mathbb{R})[\mu]$. We may assume that $\powerset(\mathbb{R})\cap L_{\alpha}(\mathbb{R})[\mu] \subseteq K(\mathbb{R})$. By the above observation, $\alpha < \Theta = \Theta^{K(\mathbb{R})}$ because otherwise, there is a surjection from $\mathbb{R}$ on $\Theta$, which contradicts the definition of $\Theta$. Now if $\powerset(\mathbb{R})\cap L_{\alpha}(\mathbb{R})[\mu] \subsetneq \powerset(\mathbb{R}) \cap K(\mathbb{R})$, then by Lemma \ref{amenability}, $\mu \rest\powerset(\mathbb{R})\cap L_{\alpha}(\mathbb{R})[\mu] \in K(\mathbb{R})$. But this means $A \in K(\mathbb{R})$. So we may assume $\powerset(\mathbb{R})\cap L_{\alpha}(\mathbb{R})[\mu] = \powerset(\mathbb{R}) \cap K(\mathbb{R})$. But this means that we can in $L_\Theta(\mathbb{R})[\mu]$ use $\mu \rest \powerset(\mathbb{R})\cap L_{\alpha}(\mathbb{R})[\mu]$ compute $\Theta^{K(\mathbb{R})}$ and this contradicts the fact that $\Theta^{K(\mathbb{R})} = \Theta$.
\end{proof}

\section{Open problems and questions}

We first mention the following
\\
\\
\indent \textbf{Conjecture: } Suppose $L(\mathbb{R})\vDash \textsf{DC} + \Theta$ is inaccessible. Then $L(\mathbb{R})\vDash \textsf{AD}$.
\\
\\
\indent This is arguably the analogous statement in $L(\mathbb{R})$ of our main theorem. It is tempting to conjecture that if $L(\mathbb{R})\vDash \rm{\textsf{DC}} + \Theta > \omega_2$ then $L(\mathbb{R})\vDash \textsf{\textsf{AD}}$ but this is known to be false by theorems of Harrington \cite{harrington1977long}. Next, we mention the following uniqueness problem which concerns the relationship between \textsf{AD} models of the form $L(\mathbb{R},\mu)$.
\\
\\
\indent \textbf{Open problem: } Suppose $L(\mathbb{R,\mu_\textrm{i}}) \vDash ``\rm{\textsf{ZF} + \textsf{DC} + \textsf{AD}} + \mu_i$ is a normal fine measure on $\powerset_{\omega_1}(\mathbb{R})$" for $i = 0,1$. Must $L(\mathbb{R},\mu_0) = L(\mathbb{R},\mu_1)$?
\\
\\
\indent We suspect that the answer is no but haven't been able to construct two distinct models of the form $L(\mathbb{R},\mu)$ that satisfy \textsf{AD}. By \rthm{mu unique}, if $L(\mathbb{R},\mu_0)$ and $L(\mathbb{R},\mu_1)$ are the same model then $\mu_0\cap L(\mathbb{R},\mu_0) = \mu_1\cap L(\mathbb{R},\mu_1)$. A generalization of the problem proved in this paper is to consider determinacy in models of the form $L(S,\mathbb{R},\mu)$ where $S$ is a set of ordinals and $L(S,\mathbb{R},\mu)\vDash ``\rm{\textsf{ZF} + \textsf{DC}} + \Theta > \omega_2 + \mu$ is a normal fine measure on $\powerset_{\omega_1}(\mathbb{R})$". Here is a (vague) conjecture.
\\
\\
\indent \textbf{Conjecture: } Let $L(S,\mathbb{R},\mu)$ be as above. Let $\Theta = \Theta^{L(S,\mathbb{R},\mu)}$ and $M_\infty$ be the maximal model of determinacy in $L(S,\mathbb{R},\mu)$. Then either $\Theta^{M_\infty} = \Theta$ or there is a model of ``$\textsf{\textsf{AD}}_\mathbb{R}+\Theta$ is regular" containing $\mathbb{R}\cup \mathrm{OR}$.
\\
\\
In another direction, we could ask about how to identify the first stage in the core model induction (under appropriate hypotheses) that reaches $\textsf{\textsf{AD}}^{L(\mathbb{R},\mu)}$ where $\mu$ comes from some filter on $\powerset_{\omega_1}(\mathbb{R})$ and $L(\mathbb{R},\mu)\vDash``\mu$ is a normal fine measure on $\powerset_{\omega_1}(\mathbb{R})$". A problem of this kind is the following
\\
\\
\indent \textbf{Open problem: }Suppose $I_{NS}$ is saturated and $WRP_2^*(\omega_2)$\footnote{$I_{NS}$ is the nonstationary ideal on $\omega_1$ and $WRP_2^*(\omega_2)$ is defined in section 9.5 of \cite{Woodin}.}. Must there be a filter $\mu$ on $\powerset_{\omega_1}(\mathbb{R})$ such that $L(\mathbb{R},\mu) \vDash ``\textsf{\textsf{AD}} + \mu$ is a normal fine measure on $\powerset_{\omega_1}(\mathbb{R})$"?
\\
\\
This problem is discussed in \cite{DFSR} though in a slightly different formulation. The point is that the hypothesis of the problem is obtained in a $\mathbb{P}_{\textrm{max}}$-extension of a model of the form $L(\mathbb{R},\mu) \vDash ``\textsf{\textsf{AD}} + \mu$ is a normal fine measure on $\powerset_{\omega_1}(\mathbb{R})$".

\bibliographystyle{plain}
\bibliography{Rmicebib}
\end{document}